\documentclass[12pt]{amsart}
 \usepackage{mathtools}
\mathtoolsset{showonlyrefs,showmanualtags} 

\usepackage{amssymb, amsmath, amsthm, esint}
\usepackage{mathrsfs}
\usepackage{bbm,dsfont}
\usepackage{hyperref}
\usepackage{mathtools}
\usepackage{fullpage}
\usepackage{color}

\usepackage[backrefs]{amsrefs}

 \usepackage[top=1.5in, bottom=1.5in, left=1.25in, right=1.25in]	{geometry}
\usepackage[all]{xy}
\usepackage{amsmath}
\usepackage{amssymb}
\usepackage{amsthm}
\usepackage{amscd}

\numberwithin{equation}{section}

\newtheorem{theorem}[equation]{Theorem}

\newtheorem{proposition}[equation]{Proposition}

\newtheorem{lemma}[equation]{Lemma}

\newtheorem{remark}[equation]{Remark}

\usepackage{hyperref} 
\hypersetup{
    colorlinks=true,       
    linkcolor=blue,          
    citecolor=magenta,        
    filecolor=magenta,      
    urlcolor=cyan           
}

\begin{document}
\title{On Multi-linear Maximal Operators Along Homogeneous Curves}

\author{Lars Becker}
\address{Mathematical Institute, University of Bonn, 
Endenicher Allee 60, 53115,
Bonn, Germany}
\email{becker@math.uni-bonn.de}

\author{Ben Krause}
\address{Department of Mathematics, University of Bristol, BS8 1QU, UK}
\email{ben.krause@bristol.ac.uk}

\date{\today}

\begin{abstract}
Suppose that
\[ \vec{\gamma}(t) := (\gamma_1(t),\dots,\gamma_n(t)) = (a_1 t^{d_1},\dots,a_n t^{d_n}), \; \; \; 1\leq d_1 < \dots < d_n, \ a_i \neq 0\]
is a homogeneous polynomial curve.
We prove that whenever $p_1,\dots,p_n > 1$ and $\frac{1}{p} = \sum_{j=1}^n \frac{1}{p_j} \leq 1$, there exists an absolute constant $0 < C = C_{p_1,\dots,p_n;\vec{\gamma}} < \infty$ so that
\[ \| \sup_{r > 0} \ \frac{1}{r} \int_{0}^r \prod_{i=1}^n |f_i(x-\gamma_i(t))| \ dt \|_{L^p(\mathbb{R})} \leq C \cdot \prod_{i=1}^n \| f_j \|_{L^{p_j}(\mathbb{R})}.
\]
%
%
Our main tool is a smoothing estimate, adapted from work of Kosz-Mirek-Peluse-Wright.

\end{abstract}

\maketitle


\section{Introduction}

The study of multi-linear maximal functions dates back to celebrated work of Lacey \cite{LAC}, who proved the following theorem.
\begin{theorem}\label{t:lac}
Suppose that $p_1,p_2 > 1$, and that $\frac{1}{p_1} + \frac{1}{p_2} = \frac{1}{p} < \frac{3}{2}$. Then there exists an absolute constant $0 < C_{p_1,p_2} < \infty$ so that
\begin{align}
    \| B^{\vec{\gamma}}(f_1,f_2) \|_{L^p(\mathbb{R})} &:= \| \sup_{r > 0} \big| \frac{1}{r} \int_0^r f(x-\gamma_1(t)) g(x-\gamma_2(t)) \ dt \big| \|_{L^p(\mathbb{R})} \\
    & \qquad \leq C_{p_1,p_2} \| f_1 \|_{L^{p_1}(\mathbb{R})} \| f_2\|_{L^{p_2}(\mathbb{R})} 
\end{align}
whenever $\gamma_i(t) = a_i t$ and $a_1 \neq a_2$ are non-zero.
\end{theorem}

The key property of the above operator is its modulation invariance, which necessitated an approach using time-frequency analysis, building off ideas of Lacey-Thiele in their work on the bilinear Hilbert transform \cite{LT1,LT2}; this method was later adapted to handle multi-linear extensions, see \cite{D}.

On the other hand, when the modulation invariance embedded in $\vec{\gamma}$ is eliminated, different techniques can be employed. This was first explored in the singular integral context in \cite{Li0,Lie0}, with subsequent work of \cite{Li} establishing the following; see also \cite{Lie}.

\begin{theorem}\label{t:Li}
    Suppose that $\gamma_1(t) = t, \ \gamma_2(t) = P(t)$, where $P(t)$ is a polynomial of degree $d$ which vanishes to degree $\geq 2$ at the origin. Then whenever $p_1,p_2 >1$, and $\frac{1}{p_1} + \frac{1}{p_2} = \frac{1}{p} < \frac{d}{d-1}$, there exists an absolute constant $0 < C_{p_1,p_2} < \infty$ so that
\[ \| B^{\vec{\gamma}}(f_1,f_2) \|_{L^p(\mathbb{R})} \leq C_{p_1,p_2} \| f_1 \|_{L^{p_1}(\mathbb{R})} \| f_2 \|_{L^{p_2}(\mathbb{R})}.\]
\end{theorem}

The key ingredient in establishing Theorem \ref{t:Li} was a \emph{Sobolev estimate}, a representative case of which is stated below.

\begin{proposition}[Special Case]\label{p:sob}
    Suppose that $\gamma_1(t) = t, \ \gamma_2(t) = t^2$, and that $\hat{f_i}$ vanishes outside $\{ |\xi| \leq 2^{l + ik} \}$ for some $i$. Then there exist absolute constants $0 < c  < C < \infty$ so that
    \begin{align}
        \| \int_0^1 f_1(x-\gamma_1(2^{-k} t)) f_2(x - \gamma_2(2^{-k} t)) \ dt \|_{L^1(\mathbb{R})} \leq C 2^{-cl} \| f_1 \|_{L^2(\mathbb{R})} \| f_2 \|_{L^2(\mathbb{R})}.
    \end{align}
\end{proposition}
In other words, the only obstruction to the estimate
\begin{align}
    \| \int_0^1 f_1(x-\gamma_1(2^{-k} t)) f_2(x - \gamma_2(2^{-k} t)) \ dt \|_{L^1(\mathbb{R})} \ll \| f_1 \|_{L^2(\mathbb{R})} \| f_2 \|_{L^2(\mathbb{R})}
\end{align}
arises from zero-frequency considerations; note that the modulation invariance from Theorem \ref{t:lac} precludes such an argument. Aside from their utility in studying the operators $\{ B^{\vec{\gamma}} \}$, Sobolev estimates have found a wide use in problems in Euclidean Ramsey theory, dating back to the work of Bourgain \cite{B}, with more recent contributions found in e.g.\ \cite{DGR, D+, Hong, K, K0}. The current state of the art for operators of the form $\{ B^{\vec{\gamma}} \}$ is essentially due to Hu-Lie \cite{HL}, who addressed trilinear formulations
\begin{align}
    \vec{\gamma} = (P_1(t),P_2(t),P_3(t))
\end{align}
provided $P_i$ are distinct degree polynomials which vanish at different rates at $0$, see Observation 1.2 (i) of \cite{HL} and \cite[Remark 2]{Lie}; their key input was a trilinear analogue of Proposition \ref{p:sob}, which we state below.

\begin{proposition}[Special Case]\label{p:sob0}
    Suppose that $\gamma_i(t) = t^i, \ 1 \leq i \leq 3$, and that $\hat{f_i}$ vanishes outside $\{ |\xi| \leq 2^{l + ik} \}$ for some $i$. Then there exist absolute constants $0 < c  < C < \infty$ so that
    \begin{align}
        \| \int_0^1 \prod_{i=1}^3 f_i(x-\gamma_i(2^{-k} t))  \ dt \|_{L^1(\mathbb{R})} \leq C 2^{-cl} \prod_{i=1}^3 \| f_i \|_{L^3(\mathbb{R})}.
    \end{align}
\end{proposition}

The goal of this paper is to address multi-linear analogues of $B^{\vec{\gamma}}$ under the simplifying assumption that our curves are homogeneous polynomials. 

Specifically, we will be concerned with
\[ \vec{\gamma}(t) := (\gamma_1(t),\dots,\gamma_n(t)) = (a_1 t^{d_1},\dots,a_n t^{d_n}), \; \; \; 1\leq d_1 < \dots < d_n, \ a_i \neq 0.\]
We prove the following.

\begin{theorem}\label{t:main}
Suppose that $p_1,\dots,p_n > 1$, and that
\[ \frac{1}{p} = \frac{1}{p_1} + \dots + \frac{1}{p_n} \leq 1.\]
Then there exists an absolute constant $0 < C_{p_1,\dots,p_n;\vec{\gamma}} < \infty$ so that 
\begin{align} \ \| \sup_{r > 0} \ \big| \frac{1}{r} \int_0^r \prod_{i=1}^n f_i(x-\gamma_i(t)) \ dt \big| \|_{L^p(\mathbb{R})}\leq C_{p_1,\dots,p_n;\vec{\gamma}} \prod_{i=1}^n \| f_i \|_{L^{p_i}(\mathbb{R})}.
\end{align}
\end{theorem}
In point of fact, the methods introduced below allow one to prove estimates slightly outside the Banach range, for all
\[ \frac{1}{p} < 1 + c_{\vec{\gamma}}, \; \; \; c_{\vec{\gamma}} > 0\]
see Remark \ref{r:subL1} below. 

As might be expected, the key input in proving Theorem \ref{t:main} is the following multi-linear Sobolev estimate:

\begin{proposition}\label{p:sob00}
    Suppose that $\gamma_i(t) = a_it^{d_i}, \ 1 \leq i \leq n$ with $1 \le d_1 < d_2< \dots< d_n, \ a_i \neq 0$, and that $\hat{f_i}$ vanishes outside $\{ |\xi| \leq 2^{l + d_ik} \}$ for some $i$. Then there exist absolute constants $0 < c  < C < \infty$ so that
    \begin{align}
        \| \int_0^1 \prod_{i=1}^n f_i(x-\gamma_i(2^{-k} t))  \ dt \|_{L^1(\mathbb{R})} \leq C 2^{-cl} \prod_{i=1}^n \| f_i \|_{L^n(\mathbb{R})}.
    \end{align}
\end{proposition}

To establish Proposition \ref{p:sob00}, we adapt a recent result of Kosz-Mirek-Peluse-Wright \cite{Kosz}; with Proposition \ref{p:sob00} in hand, Theorem \ref{t:main} readily presents. While it is reasonable to expect that an analogue of Proposition \ref{p:sob00}, and thus Theorem \ref{t:main}, should hold for more general distinct-degree polynomial curves $\gamma_i$ which vanish to distinct degrees at the origin, our argument crucially relies on homogeneity, and does not readily adapt to the more general setting. We hope to address this matter in later work.

\subsection{Notation}\label{ss:not}
Throughout, we let $\varphi$ denote various mean-one Schwartz functions, normalized in some sufficiently large semi-norm. The precise choice of $\varphi$ might differ from line to line. Similarly, we use $\psi$ to denote a similar function, but with
\begin{align}
    \mathbf{1}_{|\xi| \approx 1} \leq \hat{\psi} \leq \mathbf{1}_{|\xi| \approx 1}
\end{align}
so that
\begin{align}
    \sum_l \psi(\xi/2^l) = \mathbf{1}_{\xi \neq 0}.
\end{align}
We use the following notation to denote $L^1$-normalized dilations:
\[ \phi_k(x) := 2^{k} \phi(2^{k} x), \]
and let
\begin{align}\label{e:Bk}
    B_k^{\vec{\gamma}}(f_1,\dots,f_n)(x) := B_k(f_1,\dots,f_n)(x):= \int_0^1 \prod_{i=1}^n f_i(x-\gamma_i(2^{-k} t)) \ dt.
\end{align}

Below, we will regard $a_1,\dots, a_n = O(1)$ as arbitrary but fixed, and will abbreviate
\begin{align}\label{e:D}
D := D(\vec{\gamma}) := d_1 + \dots + d_n.
\end{align}

\subsubsection{Asymptotic Notation} We will make use of the modified Vinogradov notation. We use $X \lesssim Y$, or $Y \gtrsim X$, to denote the estimate $X \leq CY$ for an absolute constant $C$. We use $X \approx Y$ as shorthand for $Y \lesssim X \lesssim Y$. We also make use of big-O notation: we let $O(Y )$ denote a quantity that is $\lesssim Y$. We let $f(t) := o_{t \to a}(X(t))$ denote a quantity so that $\frac{|f(t)|}{X(t)} \to 0$ as $t \to a $.

If we need $C$ to depend on a parameter, we shall indicate this by subscripts, thus for instance $X \lesssim_p Y$ denotes the estimate $X \leq C_p Y$
for some $C_p$ depending on $p$. We analogously define $O_p(Y)$.

\section{Sobolev Estimates}
The main goal of this section is to prove the following Sobolev estimate.

\begin{proposition}\label{p:sob1}
There exists an absolute $c > 0$ so that the following single scale estimate holds whenever $s_i \geq 0$:
\begin{align}
    \| B_k(\psi_{kd_1+s_1}*f_1,\dots,\psi_{kd_n+s_n}*f_n) \|_{L^1(\mathbb{R})} \lesssim 2^{-cs} \prod_{i=1}^n \| f_i \|_{L^n(\mathbb{R})},
\end{align}
where $s := \max\{ s_i \}$. 
\end{proposition}

The proof of Proposition \ref{p:sob} will be accomplished via a projection argument, anchored by \cite[Theorem 6.1]{Kosz} in the case where $\gamma_i(t) = a_i t^{d_i}$, and $\mathbb{K} = \mathbb{R}$, which dictates that the below operator
    \begin{align}\label{e:A_{-k}}
        A_{-k}(F_1,\dots,F_n)(x_1,\dots,x_n) := \frac{1}{2^k} \int_0^{2^k} \prod_{i=1}^n F_i(x_1,\dots,x_i - a_i t^{d_i},\dots, x_n) \ dt
    \end{align}
satisfies non-trivial norm estimates whenever some $\widehat{F_i}$ vanishes in $|\xi_i| \leq 2^{-k d_i} \delta^{-1}$.

\begin{lemma}[Theorem 6.1 of \cite{Kosz}, Special Case]\label{l:key}
In the above setting, suppose that some $\widehat{F_i}$ vanishes in $|\xi_i| \leq 2^{-k d_i} \delta^{-1}$. Then there exists some absolute $c>0$ so that
\begin{align}\label{e:sob0}
    \| A_{-k}(F_1,\dots,F_n) \|_{L^1(\mathbb{R}^n)} \lesssim_n (\delta^c + 2^{-k c} ) \prod_{i=1}^n \| F_i \|_{L^n(\mathbb{R}^n)},
\end{align}
    provided $k \geq 0$. 
\end{lemma}
We now remove the dependence on $k$ on the right side of \eqref{e:sob0}, and address the case where $k \leq 0$. Specifically, we prove that for all $k$
\begin{align}\label{e:sob00}
    \| A_{-k}(F_1,\dots,F_n) \|_{L^1(\mathbb{R}^n)} \lesssim_n \delta^c  \prod_{i=1}^n \| F_i \|_{L^n(\mathbb{R}^n)}
\end{align}
whenever some $\widehat{F_i}$ vanishes in $|\xi_i| \leq 2^{-k d_i} \delta^{-1}$; for concreteness, suppose that this index is $j$.

To do so, introduce the operator
\begin{align}
D_\lambda F(x_1,\dots,x_n) := F(\lambda^{d_1} x_1,\dots,\lambda^{d_n} x_n)
\end{align}
and choose $2^{k_0} \gg \delta^{-1}$. If we define
\begin{align}
    G_i(x_1,\dots,x_n) := D_{2^{k-k_0}}F_i(x_1,\dots,x_n),
\end{align}
then $\widehat{G_j}$ vanishes on $|\xi_j| \leq  2^{-k_0 d_j} \delta^{-1}$, so
\begin{align}
\| A_{-k_0}(G_1,\dots,G_n) \|_{L^1(\mathbb{R}^n)} \lesssim \delta^c \prod_{i=1}^n \| G_i \|_{L^n(\mathbb{R}^n)}.
\end{align}
But now
\begin{align}
    &D_{2^{k_0-k}} \big( A_{-k_0}(G_1,\dots,G_n) \big)(x_1,\dots,x_n) \\
    &= 2^{-k_0} \int_0^{2^{k_0}} \prod_{i=1}^n G_i( 2^{d_1(k_0-k)} x_1, \dots,  2^{d_i(k_0 -k)}x_i + a_i t^{d_i},\dots, 2^{d_n(k_0-k)} x_n) \ dt \\
    & = 2^{-k} \int_0^{2^k} \prod_{i=1}^n G_i( 2^{d_1(k_0-k)} x_1, \dots,  2^{d_i(k_0 -k)}x_i + a_i 2^{d_i(k_0-k)} t^{d_i},\dots,  2^{d_n(k_0-k)} x_n) \\
    & = A_{-k} (D_{2^{k_0 -k}} G_1,\dots, D_{2^{k_0 -k}} G_n)(x) \\
    & = A_{-k} (F_1,\dots,F_n)(x),
\end{align}
so the result follows from changing variables.

With \eqref{e:sob00} in hand, we are able to prove Proposition \ref{p:sob1}.

\begin{proof}[The Proof of Proposition \ref{p:sob1}]
Set $g_i := \psi_{kd_i+s_i}*f_i$, let $\vec{1} := n^{-1/2}(1,\dots,1) \in \mathbb{R}^n$, and let $\varphi : \mathbb{R}^n \to \mathbb{C}$ be a bump function with compactly supported Fourier transform that is constant along $\mathbb{R} \vec{1}$. For each 
\[ 2^{- 100 D |k| - 100 s} \gg \epsilon > 0\] 
sufficiently small, see \eqref{e:D}, define
\begin{align}
    F_i(x_1,\dots,x_n) := \epsilon^{1-1/n} \varphi(\epsilon x) g_i(x \cdot \vec{1})
\end{align}
so that $\| F_i \|_{L^n(\mathbb{R}^n)} \approx \|g_i\|_{L^n(\mathbb{R})}$, and so that $\hat{F_i}$ is supported in an $O(\epsilon)$ neighborhood of
\begin{align}
    \{ \xi \vec{1} : \hat g_i(\xi) \ne 0 \};
\end{align}
in particular, for some $j$, $\hat{F_j}$ vanishes when $|\xi| \lesssim 2^{k d_j + s}$. So,
\begin{align}
   \| A_{k}(F_1,\dots,F_n) \|_{L^1(\mathbb{R}^n)} \lesssim 2^{-c s} \prod_{i=1}^n \| g_i \|_{L^n(\mathbb{R})}.
\end{align}
On the other hand
\begin{align}
    &\| A_{k}(F_1,\dots,F_n)(x_1,\dots,x_n) \|_{L^1(\mathbb{R}^n)} \\ &= \| 2^{k} \int_0^{2^{-k}} \big( \prod_{i=1}^n g_i(x \cdot \vec{1} - a_i t^{d_i}) \big) \cdot \big( \epsilon^{n-1} \prod_{i=1}^n \varphi(\epsilon x - \epsilon a_i t^{d_i} \vec{e_i}) \big) \ dt \|_{L^1(\mathbb{R}^n)} \\
    & = \int_{(\mathbb{R} \vec{1})^{\perp}} \int_{\mathbb{R}} |2^{k} \int_0^{2^{-k}} \big( \prod_{i=1}^n g_i(y-a_i t^{d_i}) \big) \big( \epsilon^{n-1} \prod_{i=1}^n \varphi(\epsilon z - \epsilon a_i t^{d_i} \vec{e_i} ) \big) \ dt| \ dy dz
\end{align}
using a change of variables (so in particular, $z_1,\dots,z_{n-1}$ form an orthogonal basis for $(\mathbb{R} \vec{1})^{\perp}$); above $\vec{e_i}$ is the $i$th coordinate vector.

By Taylor expansion
\begin{align}
    &\int_{(\mathbb{R} \vec{1})^{\perp}} \int_{\mathbb{R}} |2^{k} \int_0^{2^{-k}} \big( \prod_{i=1}^n g_i(y-a_i t^{d_i}) \big) \big( \epsilon^{n-1} \prod_{i=1}^n \varphi(\epsilon z - \epsilon a_i t^{d_i} \vec{e_i} ) \big) \ dt| \ dy dz \\
    &  = \| \varphi^n \|_{L^1((\mathbb{R} \vec{1})^{\perp})} \cdot \int_{\mathbb{R}^{n-1}} \int_{\mathbb{R}} |2^{k} \int_0^{2^{-k}}  \prod_{i=1}^n g_i(y-a_i t^{d_i})  \ dt| \ dy \\
    & + O( 2^{O(k)} \epsilon) \int_{\mathbb{R}^{n-1}} \Big(  \int_{\mathbb{R}} 2^{k} \int_0^{2^{-k}}  \prod_{i=1}^n |g_i(y-a_i t^{d_i})|  \ dt dy \Big) \epsilon^{n-1}( 1 + \epsilon|z|)^{-100} \ dz \\
    & = \| \varphi^n \|_{L^1((\mathbb{R} \vec{1})^{\perp})} \| B_{k}(g_1,\dots,g_n)(x) \|_{L^1(\mathbb{R})} + O_k(\epsilon \prod_{i=1}^n \| g_i \|_{L^n(\mathbb{R})});
\end{align}
the result follows from sending $\epsilon \downarrow 0$.
\end{proof}

The following proposition will be used to complement Proposition \ref{p:sob} via interpolation.

\begin{proposition}\label{p:shifted}
Suppose $n \geq 2$, $p_1,\dots,p_n > 1$ and $\frac{1}{p} = \sum_i \frac{1}{p_i} \leq 1$. Then
\begin{align}
    \| \sup_k |B_k(\psi_{k d_1+s_1}*f,\dots,\psi_{kd_n +s_n}*f)| \|_{L^p(\mathbb{R})} \lesssim s^n \prod_{i=1}^n \|f\|_{L^{p_i}(\mathbb{R})},
\end{align}
where $s := \max\{ s_i \}$.
\end{proposition}
\begin{remark}\label{r:subL1}
    In the non-Banach range, when $p < 1$, we accrue additional multiplicative factors of $2^{s(1/p-1)}$, which can be interpolated in a small neighborhood of $L^1$; we leave the details to the interested reader.
\end{remark}

We need the following lemma.

\begin{lemma}\label{l:shifted}
    Suppose that $\varphi \geq 0$ is a Schwartz function. For $t \in \mathbb{R}$, consider the maximal function
    \begin{align}
        M^tf(x):= \sup_j \varphi_j*|f|(x-2^{-j} t),
    \end{align}
    There exists an absolute constant $0 < C < \infty$ (independent of $t$) so that 
    \begin{align}
        \| M^t f \|_{L^{1,\infty}(\mathbb{R})} \leq C \log(2 + |t|) \| f \|_{L^1(\mathbb{R})}.
    \end{align}
\end{lemma}
\begin{proof}
Since the maximal function is trivially bounded on $L^{\infty}(\mathbb{R})$
we may use  vector-valued Calder\'{o}n-Zygmund theory, see e.g.\ \cite[\S 1]{FatS}. In particular, it suffices to show that
    \begin{align}
        \int_{|x| \geq 10 |y|} \sum_j |\varphi_j(x-2^{-j} t -y) - \varphi_j(x-2^{-j} t)| \lesssim \log |t|.
    \end{align}
    By dilation invariance we can normalize $|y| \approx 1$. When $2^{j} \leq 100^{-1}$ we use the mean-value theorem; when $100^{-1} \leq 2^j \lesssim |t|$ we use a single-scale estimate; and when $2^{j} \gg |t|$ we use the decay of $\varphi$.
\end{proof}

\begin{proof}[Proof of Proposition \ref{p:shifted}]
We have the pointwise bound
    \begin{align}
        &|B_k(\psi_{kd_1+s_1}*f_1,\dots,\psi_{kd_n+s_n}*f_n)(x)| \\
        &  \qquad = \big| \int_0^1 \big( \int_{\mathbb{R}^n} \prod_{i=1}^n \psi_{kd_i+s_i}(x-a_i 2^{-kd_i} t^{d_i} - u_i) f_i(u_i) \ du \big) \ dt \big| \\
        & \qquad \qquad \leq \int_0^1  \prod_{i=1}^n M^{2^{s_i} a_i t^{d_i}} f_i(x)  \ dt
    \end{align}
    So
    \begin{align}
        \| \sup_k | B_k(\psi_{kd_1+s_1}*f_1,\dots,\psi_{kd_n+s_n}*f_n)| \|_{L^p(\mathbb{R})} &\lesssim \int_0^1 \| \prod_{i=1}^n M^{2^{s_i} a_i t^{d_i}} f_i \|_{L^p(\mathbb{R})} \ dt \\
        & \lesssim s^{n} \prod_{i=1}^n \| f_i\|_{L^{p_i}(\mathbb{R})},
    \end{align}
    as desired.
\end{proof}

With Propositions \ref{p:sob1} and \ref{p:shifted} in hand, we can quickly establish Theorem \ref{t:main}.

\section{The Proof of Theorem \ref{t:main}}

The proof is by induction. Thus, we will assume that for all $m < n$
\begin{align}
    \| \sup_k |B_k(f_1,\dots,f_m)| \|_{L^p(\mathbb{R})} \lesssim \prod_{i=1}^m \| f_i \|_{L^{p_i}(\mathbb{R})},
\end{align}
whenever $\vec{\gamma}$ is a homogeneous curve, $p_1,\dots,p_m > 1$, and $\frac{1}{p_1} + \dots + \frac{1}{p_m} = \frac{1}{p} \leq 1$, with the $n=1$ case following from Hardy-Littlewood and convexity.

We first note that by Taylor expansion, for each $k$, we may decompose
\begin{align}\label{e:Tayexp}
    B_k(f_1,\dots,f_n) &= \sum_{j \geq 0} \frac{1}{j!} \big( (\partial^j \varphi)_{kd_n}*f_n\big) B_{k,\neq n}^{(j)}(f_1,\dots,f_{n-1}) \\
    & \qquad + B_k(f_1,\dots,f_{n-1},f_n - \varphi_{kd_n}*f_n) 
\end{align}
where
\begin{align}
    B_{k,\neq n}^{(j)}(f_1,\dots,f_{n-1})(x) := 2^k \int_{0}^{2^{-k}} \prod_{i=1}^{n-1} f_i(x - a_i t^{d_i}) \ (- a_n 2^{k d_n} {t^{d_n}})^j \ dt,
\end{align}
and $\partial^j \varphi$ satisfies all of the same Schwartz normalizations as $\varphi$ up to factors of $C^j, \ C = O(1)$. 

Thus, by iterating \eqref{e:Tayexp}, induction and convexity, it suffices to prove that
\begin{align}
    \| \sup_k |B_k(f_1-\varphi_{kd_1}*f_1,\dots,f_n - \varphi_{k d_n}*f_n)| \|_{L^p(\mathbb{R})} \lesssim \prod_{i=1}^n \| f_i \|_{L^{p_i}(\mathbb{R})}
\end{align}
in the above range of $p, p_i$.

Define
\begin{align}
    B_{k,s}(f_1,\dots,f_n) := \sum_{s_i \geq 0, \ \max\{s_i\} = s} B_k(\psi_{kd_1+s_1}*f_1,\dots,\psi_{kd_n+s_n}*f_n).
\end{align}
Then, by Proposition \ref{p:sob}, we may bound
\begin{align}
   \| \sup_k |B_{k,s}(f_1,\dots,f_n)| \|_{L^1(\mathbb{R})} &\leq \sum_k \| B_{k,s}(f_1,\dots,f_n) \|_{L^1(\mathbb{R})} \\
   &\leq 2^{-cs} \sum_{0 \leq s_i \leq s} \sum_k \prod_{i=1}^n \| \psi_{k d_i + s_i} *f_i \|_{L^n(\mathbb{R})} \\
   & \leq 2^{-cs} \sum_{s_i \leq s} \prod_{i=1}^n \big( \sum_{k} \| \psi_{k d_i+s_i} * f_i \|_{L^n(\mathbb{R})}^n \big)^{1/n} \\
   & = 2^{-cs} \sum_{s_i \leq s} \prod_{i=1}^n \| \big( \sum_{k} |\psi_{k d_i+s_i} * f_i|^n \big)^{1/n} \|_{L^n(\mathbb{R})}  \\
   & \leq 2^{-cs} \sum_{s_i \leq s} \prod_{i=1}^n \| S f_i \|_{L^n(\mathbb{R})}
\end{align}
where $Sf$ is the Littlewood-Paley square function,
\begin{align}
Sf := \big( \sum_k |\psi_k*f|^2 \big)^{1/2},
\end{align}
which is bounded on $L^p(\mathbb{R}), \ 1 < p < \infty$. So
\begin{align}\label{e:sum}
    \sum_k \| B_{k,s}(f_1,\dots,f_n) \|_{L^1(\mathbb{R})} \lesssim_n s^n 2^{-cs} \prod_{i=1}^n \| f_i \|_{L^n(\mathbb{R})},
\end{align}
and by interpolating with Proposition \ref{p:shifted}, we see that whenever $p_1,\dots,p_n > 1, \ \frac{1}{p} = \sum_i \frac{1}{p_i} \leq 1$, there exists an absolute $c = c_{p_1,\dots,p_n;p} > 0$ so that
\begin{align}
    \| \sup_k |B_{k,s}(f_1,\dots,f_n)| \|_{L^p(\mathbb{R})} \lesssim s^n 2^{-c s} \prod_{i=1}^n \| f_{i} \|_{L^{p_i}(\mathbb{R})}.
\end{align}
A final sum over $s \geq 1$ completes the proof.

\end{document}